\documentclass{amsart}
\usepackage{amsfonts,amsmath,amsthm,amssymb, color}
\usepackage[all]{xy}
\usepackage{tikz-cd}
\usepackage[
colorlinks=true, linkcolor=blue, citecolor=blue]{hyperref}

\numberwithin{equation}{section}

\begin{document}

\newtheorem{theorem}{Theorem}[section]
\newtheorem{lemma}[theorem]{Lemma}
\newtheorem{proposition}[theorem]{Proposition}
\newtheorem{corollary}[theorem]{Corollary}

\theoremstyle{definition}
\newtheorem{definition}[theorem]{Definition}
\newtheorem{example}[theorem]{Example}

\theoremstyle{remark}
\newtheorem{remark}[theorem]{Remark}

\newenvironment{magarray}[1]
{\renewcommand\arraystretch{#1}}
{\renewcommand\arraystretch{1}}

\newcommand{\mapor}[1]{\smash{\mathop{\longrightarrow}\limits^{#1}}}
\newcommand{\mapin}[1]{\smash{\mathop{\hookrightarrow}\limits^{#1}}}
\newcommand{\mapver}[1]{\Big\downarrow
\rlap{$\vcenter{\hbox{$\scriptstyle#1$}}$}}
\newcommand{\liminv}{\smash{\mathop{\lim}\limits_{\leftarrow}\,}}

\newcommand{\Set}{\mathbf{Set}}
\newcommand{\Art}{\mathbf{Art}}
\newcommand{\solose}{\Rightarrow}

\renewcommand{\bar}{\overline}
\newcommand{\de}{\partial}
\newcommand{\debar}{{\overline{\partial}}}
\newcommand{\per}{\!\cdot\!}
\newcommand{\Oh}{\mathcal{O}}
\newcommand{\sA}{\mathcal{A}}
\newcommand{\sB}{\mathcal{B}}
\newcommand{\sC}{\mathcal{C}}
\newcommand{\sE}{\mathcal{E}}
\newcommand{\sF}{\mathcal{F}}
\newcommand{\sG}{\mathcal{G}}
\newcommand{\sH}{\mathcal{H}}
\newcommand{\sI}{\mathcal{I}}
\newcommand{\sJ}{\mathcal{J}}
\newcommand{\sL}{\mathcal{L}}
\newcommand{\sM}{\mathcal{M}}
\newcommand{\sP}{\mathcal{P}}
\newcommand{\sU}{\mathcal{U}}
\newcommand{\sV}{\mathcal{V}}
\newcommand{\sX}{\mathcal{X}}
\newcommand{\sY}{\mathcal{Y}}

\newcommand{\Aut}{\operatorname{Aut}}
\newcommand{\Mor}{\operatorname{Mor}}
\newcommand{\Def}{\operatorname{Def}}
\newcommand{\Hom}{\operatorname{Hom}}
\newcommand{\Alt}{\operatorname{Alt}}
\newcommand{\Tot}{\operatorname{Tot}}
\newcommand{\HOM}{\operatorname{\mathcal H}\!\!om}
\newcommand{\DER}{\operatorname{\mathcal D}\!er}
\newcommand{\Spec}{\operatorname{Spec}}
\newcommand{\Der}{\operatorname{Der}}
\newcommand{\Tor}{{\operatorname{Tor}}}
\newcommand{\Ext}{\operatorname{Ext}}
\newcommand{\Sym}{\operatorname{Sym}}
\newcommand{\End}{{\operatorname{End}}}
\newcommand{\END}{\operatorname{\mathcal E}\!\!nd}
\newcommand{\Image}{\operatorname{Im}}
\newcommand{\coker}{\operatorname{coker}}
\newcommand{\Id}{\operatorname{Id}}
\newcommand{\ad}{\operatorname{ad}}
\newcommand{\Span}{\operatorname{Span}}

\newcommand{\ghirigoro}{\rightsquigarrow}
\renewcommand{\Hat}[1]{\widehat{#1}}
\newcommand{\dual}{^{\vee}}
\newcommand{\desude}[2]{\dfrac{\de #1}{\de #2}}

\newcommand{\A}{\mathbb{A}}
\newcommand{\N}{\mathbb{N}}
\newcommand{\R}{\mathbb{R}}
\newcommand{\Z}{\mathbb{Z}}
\renewcommand{\H}{\mathbb{H}}
\newcommand{\C}{\mathbb{C}}
\newcommand{\proj}{\mathbb{P}}
\newcommand{\K}{\mathbb{K}}


\newcommand{\contr}{{\mspace{1mu}\lrcorner\mspace{1.5mu}}}
\newcommand{\wedgebar}{\;{\scriptstyle\mathchar'26\mkern-10.5mu\wedge}\;}

\newcommand{\bi}{\boldsymbol{i}}
\newcommand{\bl}{\boldsymbol{l}}

\newcommand{\MC}{\operatorname{MC}}
\newcommand{\Coder}{{\operatorname{Coder}}}


\title{Some examples of DG-Lie formality transfer}

\author{Marco Manetti}
\address{\newline
Universit\`a degli studi di Roma La Sapienza,\hfill\newline
Dipartimento di Matematica Guido Castelnuovo,\hfill\newline
P.le Aldo Moro 5, I-00185 Roma, Italy.}
\email{manetti@mat.uniroma1.it}
\urladdr{sites.google.com/uniroma1.it/marcomanetti/}

\author{Gabriele Rossetti}
\email{gabriele.rossetti2809@gmail.com}

\begin{abstract} We prove a splitting critera for  the formality transfer theorem for DG-Lie algebras and we give some geometric applications.
\end{abstract}

\maketitle

\section{Introduction}

The goal of this paper is to give a proof and some applications of  the following formality transfer criterion.

\begin{theorem}[formality transfer]\label{thm.main} 
Let $f\colon L\to M$ be a morphism of DG-Lie algebras and consider the 
$L$-module structure on $M$ induced by $f$.

If $M$ is formal, $f$ is injective and $f(L)$ is a direct summand of $M$ as differential graded $L$-module,  then  $L$ is also formal.
\end{theorem}

This theorem has already been used several times in the literature. For instance, it was used in the proof of the Kaledin--Lehn formality conjecture for polystable sheaves on Enriques and bielliptic surfaces \cite{BMM2}.  
In the previous literature, as well as in the first preprint version of this paper, the
reference for  Theorem~\ref{thm.main} 
was \cite{Ma15}, since it is an immediate consequence of Theorem~6.8 therein.
Unfortunately, that theorem  relies on \cite[Lemma~5.3]{Ma15} that contains a mistake in the proof.
Moreover, the example of Subsection~\ref{subsec.nonexample} of this paper shows that \cite[Theorem~6.8]{Ma15} is generally false without additional assumptions (see the appendix of the arXiv version v3 of \cite{Ma15} for a discussion of the mistake and for corrigenda).

The first goal of this paper is to  provide  a direct  proof of the above theorem, based on \cite[Theorem~3.3]{Ma15}, that is not affected by the above mentioned mistake, since in this case the lemma is applied in the particular case $f=\Id$, where it is true for trivial reasons. 

Then 
we  describe some  applications of formality transfer. In particular, we prove Theorem~\ref{thm.formalityUL}  about the equivalence of Lie and associative formality for universal enveloping algebras, and Theorem~\ref{thm.KSunderaction},  about the persistence of Kodaira--Spencer formality under free actions of finite groups.

\bigskip
\section{Basic facts about DG-Lie formality}
\label{sec.simple}

Throughout the whole paper, every differential graded (DG) vector space is intended to be defined over a field
$\K$ of characteristic 0, is graded over the integers,  and has a differential of degree $+1$. Given a homogeneous vector  $x$ in a DG-vector space, we shall indicate its degree by $\bar{x}$. We assume that the reader is also familiar with the notion of (differential) graded Lie algebra.

If $L$ is a DG-Lie algebra, then its cohomology $H^*(L)$ inherits a natural structure of DG-Lie algebra with trivial differential. By definition, $L$ is called formal  if there is a zigzag of quasi-isomorphisms of DG-Lie algebras connecting $L$ with  $H^*(L)$.

The notion of formality of a differential graded Lie algebra  has received considerable attention following  the papers of Goldman, Millson \cite{GoMil1} and Kontsevich \cite{K}. 
In  \cite{GoMil1} the authors proved the 
formality of the differential graded Lie algebra of differential forms with values in certain flat bundles of Lie algebras; as a   consequence of this fact they proved that the moduli space of certain representations of the fundamental  group of a   compact K\"{a}hler manifold has at most quadratic singularities. The key point is that  the deformation functor associated to a DG-Lie algebra $L$ is invariant under quasi-isomorphism and therefore it takes a very simple form whenever $L$ is formal. 

In the paper \cite{K}  Kontsevich proved that, when $A$ is the algebra of smooth functions on a differentiable 
manifold, then the natural DG-Lie algebra structure on the 
Hochschild cohomology complex of polydifferential operators of $A$ is formal, thereby   proving that every finite-dimensional Poisson manifold admits a deformation quantization. 
More recent results about formality of DG-Lie algebras and applications to deformation theory include \cite{BMM2,BMM1,FMpoisson,indam,MartPadova,MO}.

The following standard lemma  gives some useful equivalent conditions for formality. 

\begin{lemma}\label{lemma.condizioniformalita} 
For a DG-Lie algebra $L$ over a field of characteristic 0, the following conditions are equivalent:
\begin{enumerate}
\item $L$ is formal;
\item there exists a span $L\leftarrow M\rightarrow H^*(L)$ of surjective quasi-isomorphisms of DG-Lie algebras;

\item there exists an $L_{\infty}$-morphism $L\ghirigoro H^*(L)$ extending the identity in cohomology.

\end{enumerate}
\end{lemma}

\begin{proof} This is a well-known result. For a complete proof one can see, for instance, the book \cite{LMDT}.
\end{proof}

A first application of the above lemma is the following  transfer 
criteria. Recall that a DG-Lie algebra is called homotopy abelian if it is quasi-isomorphic to an abelian DG-Lie algebra. 

\begin{proposition}\label{prop.cheap} 
Let $f\colon L\to M$ be a morphism of DG-Lie algebras.
\begin{enumerate}

\item if $M$ is homotopy abelian and $f\colon H^*(L)\to H^*(M)$ has a left inverse in the category of graded vector spaces, then also $L$ is homotopy abelian;

\item if $M$ is formal and $f\colon H^*(L)\to H^*(M)$ has a left inverse in the category of graded Lie algebras, then  $L$ is also formal.

\end{enumerate}
\end{proposition}

\begin{proof} The first item is well known and widely used in the literature; 
the reader can see  \cite[Cor. 6.1.3]{LMDT} and \cite{BaMacoisotropic,FMpoisson,algebraicBTT, KKP} for various proofs and applications to deformation theory.

As regards the second item, by Lemma~\ref{lemma.condizioniformalita} there exists an $L_\infty$-morphism  $M\ghirigoro H^*(M)$ extending the identity in cohomology and a morphism of graded Lie algebras $g\colon H^*(M)\to H^*(L)$ such that $gf=\Id$. Then the composition $L\xrightarrow{f}M\ghirigoro H^*(M)\xrightarrow{g}H^*(L)$ is an  $L_\infty$-morphism that extends the identity in cohomology.
\end{proof}

Unfortunately, for now, the second item of Proposition~\ref{prop.cheap} does not seem to provide any interesting geometric application beyond the ones already included in the first item.

\bigskip
\section[The Chevalley--Eilenberg double complex]{The Chevalley--Eilenberg double complex and proof of Theorem~\ref{thm.main}}
\label{sec.prova}

It is well understood, see for instance \cite{halsta,hinich,kaledin,lunts,NijRich67,Sa17}, that the right  context for obstructions to formality of DG-Lie algebras is Chevalley--Eilenberg cohomology theory; the classical (non-graded) definition \cite{CE,HS,Ja} generalizes easily to the differential graded case and gives cohomology groups $H_{CE}^p(L,M)$, where $L$ is a graded Lie algebra and $M$ is a graded $L$-module, see  Definition~\ref{def.CEcohomology} below.\medskip

Let $(L,d,[-,-])$ be a DG-Lie algebra. By definition, an $L$-module is a DG-vector space $M$ equipped with a morphism of DG-vector spaces
\[
M\otimes L\xrightarrow{\;*\;} M,\qquad m\otimes x\mapsto m*x,
\]
such that
\[
m*[x,y]=(m*x)*y-(-1)^{\bar{x}\,\bar{y}}(m*y)*x.
\]
For instance, if $f\colon L\rightarrow M$ is a morphism of DG-Lie algebras, then $M$ is an $L$-module via the adjoint representation $m*x=[m,f(x)]$.

Given a DG-Lie algebra $L$ and an $L$-module $M$, we recall from \cite{Ma15} the construction of the associated Chevalley--Eilenberg double complex $(CE(L,M)^{*,*}, \delta,\bar{\delta})$.

We have:

\begin{enumerate}
\item  $CE(L,M)^{p,q}= \Hom^q_\K(L^{\wedge p}, M)$; in particular  
$CE(L,M)^{p,q}=0$ for $p<0$ and $CE(L,M)^{0,q}=M^q$.

\item  $\bar{\delta}\colon CE(L,M)^{p,q}\rightarrow CE(L,M)^{p,q+1}$ is the natural differential of the complex $\Hom^*_\K(L^{\wedge p}, M)$, that is 
    \[
    (\bar{\delta}\phi)(x_1,\dots,x_p)=d(\phi(x_1,\dots,x_p))-\sum_{i=1}^{p}(-1)^{\bar{\phi}+\bar{x_1}+\dots+\bar{x_{i-1}}}\phi(x_1,\dots,dx_i,\dots,x_p)
    \]
\item $\delta$ is the natural extension of the classical Chevalley--Eilenberg differential (with
an inessentially different choice of signs motivated by the d\'ecalage isomorphisms, cf.  \cite[Section 7]{Ma15}). More precisely:

\begin{enumerate}
\item for  $m\in CE(L,M)^{0,q}=M^q$ we have $(\delta m)(x)=(-1)^{q}m*x$;\smallskip

\item  for $\phi\in CE(L,M)^{1,q}=\Hom_{\K}^q(L,M)$ we have 
\[(\delta\phi)(x,y)=(-1)^{q+1}\left(\phi(x)*y-
(-1)^{\bar{x}\;\bar{y}}\phi(y)*x-\phi([x,y])\right);\]
\smallskip

\item for $p\ge 2$ and  $\phi\in CE(L,M)^{p-1,q}=\Hom_{\K}^q(L^{\wedge p-1},M)$  we have:
\[ \begin{split}
(\delta \phi)(x_1,\ldots,x_p)&=
(-1)^{q+p-1}\sum_{i}\chi_i\, \phi(x_1,\ldots,\widehat{x_i},\ldots,x_p)*x_i\\
&\quad+(-1)^{q+p}
\sum_{i<j}\chi_{i,j}\, \phi(x_1,\ldots,\widehat{x_i},\ldots,\widehat{x_j},\ldots,x_p,[x_i,x_j]),
\end{split}\]

where the  $\chi_i,\chi_{i,j}\in \{\pm 1\}$ are the antisymmetric Koszul signs,  determined by the equalities in 
$L^{\wedge p}$:
\[ x_1\wedge\cdots\wedge x_p=\chi_i\;
x_1\wedge\cdots\widehat{x_i}\cdots\wedge x_p\wedge x_i=\chi_{i,j}\;
 x_1\wedge\cdots\widehat{x_i}\cdots\widehat{x_j}\cdots\wedge x_p\wedge x_i\wedge x_j\;.\]
\end{enumerate}
\end{enumerate}

\begin{definition}\label{def.CEcohomology} Given $L,M$ as above, the Chevalley--Eilenberg cohomology of $L$ with coefficients in $M$ is the cohomology of the total complex $\Tot^{\Pi}(CE(L,M))$ 
(see \cite[1.2.6]{weibel}):
\[ H^*_{CE}(L,M)=H^*(\Tot^{\Pi}(CE(L,M))),\quad \Tot^{\Pi}(CE(L,M))^n=\prod_{p+q=n}CE(L,M)^{p,q}.\]
\end{definition}

It is worth  pointing out that, when both $L$ and $M$ have trivial differential, then $\Tot^{\Pi}(CE(L,M))$ is the direct product of the subcomplexes $(CE(L,M)^{*,q},\delta)$.

We denote by $(E(L,M)_r^{p,q},d_r)$ the cohomological spectral sequence associated to the first filtration of the double complex $(CE(L,M)^{*,*}, \delta,\bar{\delta})$: more precisely 
\[ F^pCE(L,M)^{*,*}=\prod_{i\ge p}CE(L,M)^{i,*}.\]

In particular, 
\[E(L,M)_0^{p,q}=\frac{F^p\Tot^{\Pi}(CE(L,M))^{p+q}}{F^{p+1}\Tot^{\Pi}(CE(L,M))^{p+q}}=\Hom_{\K}^q(L^{\wedge p},M),\qquad d_0=\bar{\delta}.\] By the K\"{u}nneth formula
there exist  natural isomorphisms 
\[ E(L,M)_1^{p,q}\xrightarrow{\simeq}\Hom_{\K}^q(H^*(L)^{\wedge p},H^*(M))=E(H^*(L),H^*(M))_1^{p,q}
,\qquad d_1=\delta,\]
and therefore 
\[ H_{CE}^n(H^*(L),H^*(M))\cong \prod_{p\ge 0}E(L,M)_2^{p,n-p}.\]
In particular, \[   E(L,M)^{1,*}_2=E(H^*(L),H^*(M))^{1,*}_2=
\frac{\{\text{derivations}\;  H^*(L)\to H^*(M)\}}{\{\text{inner derivations}\}}.\]

\medskip

If $f\colon L\to M$ is a morphism of DG-Lie algebras, then  a particular element 
$e_f\in E(L,M)^{1,0}_2$, called the \emph{Euler class}, is defined. By definition, $e_f$ is the class of the Euler derivation 
\[ \epsilon_f\colon H^*(L)\to H^*(M),\qquad \epsilon_f(x)=\bar{x}\, f(x),\]
modulo inner derivations. Equivalently $e_f\in E(L,M)^{1,0}_2$ is the $d_1$-cohomology class of 
\[\epsilon_f\in E(L,M)_1^{1,0}=\Hom^0_{\K}(H^*(L),H^*(M)),\qquad \epsilon_f(x)=\bar{x}\, f(x).\]

Whenever $d_2(e_f)=0$ we continue to denote by $e_f\in E(L,M)^{1,0}_3$ its cohomology class; recursively, if $d_2(e_f)=\cdots=d_r(e_f)=0$ for some $r\ge 2$ we still denote by  $e_f\in E(L,M)^{1,0}_{r+1}$ the class of the Euler derivation.

Finally, both the spectral sequence and the Euler class behave functorially with respect to morphisms of DG-Lie algebras. This means that, given two morphisms $L\xrightarrow{f}M\xrightarrow{g}N$ of DG-Lie algebras, we have two morphisms of spectral sequences 
\[ E(L,M)_*\xrightarrow{\;g_*\;} E(L,N)_*\xleftarrow{\;f^*\;}E(M,N)_*,\]
such that $g_*(e_f)=e_{gf}=f^*(e_g)$.
According to the previous description of the page $E(L,N)_1$,  
if $g$ (resp.: $f$) is a quasi-isomorphism, then $g_*$ (resp.: $f^*$) is an isomorphism at the  page
$E(L,N)_r$, for every $r\ge 1$.

When $f$ is the identity map on a DG-Lie algebra $L$ we simply write
\[ e_L\in E(L,L)^{1,0}_2=\frac{\;\Der^0_{\K}(H^*(L),H^*(L))\;}{
\{[x,-]\mid x\in H^0(L)\}},\qquad e_L(x)=\bar{x}\cdot x,\]
instead of $e_{\Id_L}$. It is essential to observe that, while $e_L$ depends only on the 
graded Lie algebra  $H^*(L)$, its differential  
\[d_2(e_L)\in E(L,L)_2^{3,-1}\subseteq H_{CE}^2(H^*(L),H^*(L))\] 
depends on $L$. In fact, the vanishing of $d_2(e_L)$ is essentially equivalent to the vanishing of all triple Massey powers \cite{simi}.

\begin{theorem}[{\cite[Theorem~3.3]{Ma15}}]\label{thm.formalitycriteria} 
Let $(E(L,L)^{p,q}_r,d_r)$
be the Chevalley--Eilenberg spectral sequence of  a differential graded Lie algebra $L$. 
Then the following conditions are equivalent:

\begin{enumerate}

\item\label{it1.thm.formalitycriteria} $L$ is formal;

\item the spectral sequence $E(L,L)^{p,q}_r$ degenerates at $E_2$;

\item\label{it3.thm.formalitycriteria} $d_r(e_L)=0\in E(L,L)^{r+1,1-r}_r$ for every $r\ge 2$.
\end{enumerate}
\end{theorem}

We are now ready to prove the formality transfer criterion.

\begin{proof}[Proof of Theorem~\ref{thm.main}]
Let $f\colon L\to M$ be an injective morphism of DG-Lie algebras such that $M$ is formal and 
$M=f(L)\oplus N$ as differential graded $L$-modules.
By Theorem~\ref{thm.formalitycriteria} we have $d_r(e_M)=0$ for every $r\ge 2$ and then 
$d_r(e_f)=d_r(f^*e_M)=f^*(d_r(e_M))=0$ for every $r\ge 2$.

Let $r\ge 2$ and assume by induction that $d_i(e_L)=0$ for every $2\le i<r$. 
Then $f_*(d_r(e_L))=d_r(e_f)=0$. On the other hand, since $M$ is the direct sum of the DG $L$-modules $f(L)$ and $N$, we have 
$CE(L,M)=CE(L,f(L))\oplus CE(L,N)$ as filtered double complexes.

Therefore the spectral sequence $E(L,L)^{p,q}_r$ is a direct summand of 
$E(L,M)^{p,q}_r$; in particular, the map $f_*\colon E(L,L)_r^{r+1,1-r}\to 
E(L,M)^{r+1,1-r}_r$ is injective and then $d_r(e_L)=0$.
\end{proof}

It is worth  mentioning the existence of a different, and also useful, criterion for formality based on the study of gauge equivalence of the Kaledin class;  we refer to the original paper by Kaledin \cite{kaledin}, see also \cite{lunts}, as well as the more recent contributions \cite{emprin1,emprin}.

\bigskip
\section{Examples and applications}

\subsection{Universal Enveloping Algebras}
For every DG-Lie algebra $L$ we denote by $U(L)$ its universal enveloping algebra; it is defined, mutatis mutandis, as in the classical case, i.e., as the quotient $T(L)/I$ where
$T(L)$ is the DG-associative tensor algebra generated by the differential graded vector space underlying $L$, and $I$ is the ideal generated by elements of type $x\otimes y-(-1)^{\bar{x}\,\bar{y}\,}y\otimes x-[x,y]$.

\begin{theorem}\label{thm.formalityUL}
For a DG-Lie algebra $L$ the following properties are equivalent:
\begin{enumerate}
\item $L$ is formal;
\item $U(L)$ is formal as differential graded associative algebra;
\item $U(L)$ is formal as differential graded Lie algebra.
\end{enumerate}
\end{theorem}

The equivalence of the first two items in Theorem~\ref{thm.formalityUL} was proved, in a different way, by Saleh in \cite{Sa17}.

\begin{proof} The implication $2)\Rightarrow 3)$ is obvious and the 
implication $1)\Rightarrow 2)$ follows easily from the existence of the natural isomorphism 
of graded associative algebras $H^*(U(L))=U(H^*(L))$ \cite[B.2.1]{Qui}.

The proof of the implication  $3)\Rightarrow 1)$  relies on the 
Poincar\'e--Birkhoff--Witt (PBW) theorem. If $S(L)=\oplus_{n=0}^\infty L^{\odot n}$ denotes the symmetric DG-algebra generated by $L$ and $i\colon L\to U(L)$ is the natural map, the PBW theorem \cite[B.2.3]{Qui} asserts in particular that the linear map 
 \[
    e\colon S(L)\xrightarrow{\;\simeq\;} U(L),\quad e(x_1\odot\dots\odot x_n)=\frac{1}{n!}\sum_{\sigma\in\Sigma_n}\epsilon(\sigma)\,i(x_{\sigma(1)})\cdots i(x_{\sigma(n)}),
    \]
where $\epsilon(\sigma)$ is the Koszul sign and $\Sigma_n$ the symmetric group, is an isomorphism of DG-vector spaces.   It is a well-known result that $e$ is also an isomorphism of graded coalgebras, although this property is not relevant in this proof.

Given  $x\in L$, we denote by $r_x\colon S(L)\to S(L)$ the unique derivation of the symmetric algebra such that $r_x(y)=[x,y]$ for every $y\in L$.   A straightforward computation, see \cite[Lemma 3.3.5]{Loday} for further details, shows that 
for every $x\in L$ and every $z\in S(L)$ we have $e(r_x(z))=[i(x),e(z)]$.

This implies that $H:=e(\oplus_{n\not=1}L^{\odot n})$ is an $L$-submodule of $U(L)$ and we have 
a direct sum  decomposition 
$U(L)=i(L)\oplus H$. According to Theorem~\ref{thm.main}, if $U(L)$ is formal as a DG-Lie algebra, then also $L$ is formal.
\end{proof}

\subsection{Invariant subalgebras}

We begin with an equivalent formulation of  Theorem~\ref{thm.main}.

\begin{corollary}\label{cor.main3}
Let $M$ be a formal DG-Lie algebra and let $L\subseteq M$ be a DG-Lie subalgebra. Assume that there exists a retraction of $L$-modules $p\colon M\to L$, i.e. a 
morphism of differential graded vector spaces  $p\colon M\to L$ such that 
\begin{equation}\label{equ.proiezione} 
p(x)=x,\quad p([x,y])=[x,p(y)]\quad\text{ for every }\quad x\in L,\, y\in M.\end{equation}
Then also $L$ is formal. 
\end{corollary}

\begin{proof}
We have  
$M=L\oplus \ker p$ with $[L,\ker p]\subseteq \ker p$, and the formality of $L$ follows from Theorem~\ref{thm.main}.
\end{proof}

A typical situation where Corollary~\ref{cor.main3} applies  is when $L=M^G$ is the subalgebra of invariant elements under a finite group $G$ of automorphisms of $M$, and $p$ is given by the usual Reynolds
operator
\[ p(x)=\frac{1}{|G|}\sum_{g\in G}g(x).\]

More generally, the same argument applies when $G$ is a linearly reductive algebraic group acting on $M=\oplus M^i$ and every $M^i$ is a finite direct sum $M^i=\oplus_j U_j\otimes V_j$, where $G$ acts trivially on every $U_j$ and every $V_j$ is a rational finite-dimensional representation of $G$, see \cite{BMM1}.

A geometric example of this situation is described in \cite[Section~5]{BMM2}, where
$L$ is the DG-Lie algebra of derived endomorphisms of a polystable sheaf  on a surface with torsion canonical bundle.

\subsection{Quotients by finite free regular actions}

Let $X$ be a smooth compact complex manifold or a smooth projective  manifold over an algebraically closed field of characteristic 0, and let $L=KS_X$ be its Kodaira--Spencer DG-Lie algebra; recall that 
$KS_X$ is the homotopy class of the DG-Lie algebra controlling infinitesimal deformations of $X$ (see e.g. \cite{LMDT}) and it is represented  by:
\begin{enumerate}
\item (complex case) the Dolbeault complex of the holomorphic tangent sheaf $A^{0,*}(\Theta_X)$, or
\item (algebraic case) the Thom--Whitney totalization of  the semicosimplicial  Lie
algebra of \v{C}ech cochains of the tangent sheaf over an affine open cover.
\end{enumerate}

Following \cite{indam}, we say that $X$ is Kodaira--Spencer formal if the DG-Lie algebra 
$KS_X$ is formal; by standard results, see e.g. \cite{LMDT}, if  $X$ is Kodaira--Spencer formal, then the base germ of its semiuniversal deformation is a quadratic singularity.
In the complex case, it is worth recalling  the main result of \cite{indam}:  if $X,Y$ are compact, K\"{a}hler and
Kodaira--Spencer formal complex manifolds, then also the product $X\times Y$ is 
Kodaira--Spencer formal; for instance,
the product 
$X=T\times \mathbb{P}^r$, with $T$ a complex torus, 
is Kodaira--Spencer formal (when $r\ge 1$ and $\dim T\ge 2$, Kodaira and Spencer proved that $X$ has obstructed deformations \cite[p. 436]{KS}).

\begin{theorem}\label{thm.KSunderaction} 
Let $X$ be a smooth projective manifold over an algebraically closed field 
of characteristic 0 and let $G$ be a finite group acting 
freely on $X$ by regular automorphisms. If $X$ is Kodaira--Spencer formal, then also the quotient $X/G$ is Kodaira--Spencer formal.
\end{theorem}

\begin{proof}  We refer to \cite{LMDT} for the definition and the main properties of the Thom--Whitney totalization functor.

Set $Y=X/G$; 
by  standard  facts in algebraic geometry, the variety $Y$ is smooth and projective and the quotient map $\pi\colon X\to Y$ is an \'etale finite morphism.  In particular, if 
$\Theta_X$ and $\Theta_Y$ are the tangent sheaves of $X$ and $Y$ respectively, then 
$\pi^*\Theta_Y=\Theta_X$ and 
for every open subset  $U\subseteq Y$ there is a well-defined injective map 
$\Gamma(U,\Theta_Y)\to \Gamma(\pi^{-1}(U),\Theta_X)$ whose image is the subspace 
$\Gamma(\pi^{-1}(U),\Theta_X)^G$ of $G$-invariant vector fields. It is important to observe that $G$ acts on $\Gamma(\pi^{-1}(U),\Theta_X)$ by Lie algebra isomorphisms.

Taking an affine open cover $\sU=\{U_i\}$ of $Y$, then $\sV=\{\pi^{-1}(U_i)\}$ is an affine open cover of $X$ and the respective 
Thom--Whitney 
totalizations $\Tot(\sU,\Theta_Y)$ and $\Tot(\sV,\Theta_X)$ are 
DG-Lie algebras representing $KS_Y$ and $KS_X$, respectively.

Since the action of $G$ on the vector spaces $\Gamma(\pi^{-1}(U),\Theta_X)$ commutes with restriction to open subsets, the group $G$ acts also on 
$\Tot(\sV,\Theta_X)$ and
$\Tot(\sU,\Theta_Y)=\Tot(\sV,\Theta_X)^G$. 
 Thus we can apply Corollary~\ref{cor.main3}, with $p$ the Reynolds operator.
\end{proof}

The same result holds also in the compact holomorphic case, because
$A^{0,*}(\Theta_Y)=A^{0,*}(\Theta_X)^G$.

A classical construction where Theorem~\ref{thm.KSunderaction} applies is the following. Consider a finite group $G$ acting freely on a complex abelian variety $A$ by regular isomorphisms and a representation $G\to GL(\C^{r+1})$; then $G$ acts diagonally and freely on the product $X=A\times \mathbb{P}^r$, and the quotient $Y$ is a 
$\mathbb{P}^r$-bundle over $A/G$. Then $Y$ is Kodaira--Spencer formal; hence $Y$ has unobstructed deformations if and only if the bracket on $H^1(Y,\Theta_Y)$ vanishes.

Some results similar to Theorem~\ref{thm.KSunderaction}, concerning  homotopy abelianity and unobstructedness of the associated deformation functor, were already known;  see for instance 
\cite[Ex. 15]{CCK}
and \cite[Thm. 6.2]{algebraicBTT}.

\subsection{Formality of direct summand of coherent sheaves}

Here is a very easy and immediate application of Theorem~\ref{thm.main}; it is written mainly for reference purposes. 

Let $\sF$ be a coherent sheaf on a separated scheme $X$ of finite type over a field of characteristic $0$. It is well known that 
the deformations of $\sF$ are controlled by the DG-Lie algebra of derived endomorphisms 
$R\Hom(\sF,\sF)$. Recall that $R\Hom(\sF,\sF)$ is  defined up to quasi-isomorphism and a
representative is given by   
$\Hom_{\Oh_X}(\sI^*,\sI^*)$, where $0\to \sF\to \sI^*$ is an 
injective resolution in the category of $\Oh_X$-modules.  This is a well known result, see for instance \cite[Section 2.A.6]{HL} and \cite[Section 2.5.2]{tesilepri}.

\begin{corollary} 
Let $\sF,\sG$ be coherent sheaves on a separated scheme $X$ of finite type over a field of characteristic $0$. If the DG-Lie algebra $R\Hom(\sF\oplus\sG,\sF\oplus\sG)$ is formal, then also the DG-Lie algebras $R\Hom(\sF,\sF)$ and $R\Hom(\sG,\sG)$ are formal.
\end{corollary}

\begin{proof} By symmetry it is sufficient to prove the formality of $R\Hom(\sF,\sF)$.
Choose two injective resolutions $\sI^*$ and $\sJ^*$ of $\sF$ and $\sG$ respectively. Then $\sI^*\oplus \sJ^*$ is an injective resolution of $\sF\oplus \sG$. 
Then we can write 
\[\Hom_{\Oh_X}(\sI^*\oplus\sJ^*,\sI^*\oplus\sJ^*)=L\oplus N,\]
where $L\cong \Hom_{\Oh_X}(\sI^*,\sI^*)$ is the Lie subalgebra of endomorphisms $f$ such that $f(\sI^*)\subset \sI^*$, $f(\sJ^*)=0$, and $N$ is the $L$-submodule of endomorphisms $g$ such that $g(\sI^*)\subseteq \sJ^*$. 

According to Theorem~\ref{thm.main}, if $\Hom_{\Oh_X}(\sI^*\oplus\sJ^*,\sI^*\oplus\sJ^*)$ is formal, then  also $\Hom_{\Oh_X}(\sI^*,\sI^*)$ is formal.
\end{proof}

\subsection{A non-example of formality transfer}\label{subsec.nonexample}

Theorem~\ref{thm.main} fails if we only assume that the induced morphism 
$f\colon H^*(L)\to H^*(M)$ is split injective in the category of graded 
$H^*(L)$-modules; this contradicts the validity of \cite[Theorem~6.8]{Ma15}, since 
the splitting a $H^*(L)$-modules implies the injectivity of $E(L,L)_2^{*,*}\to E(L,M)_2^{*,*}$.

For example, let $M$ be the DG-Lie algebra with basis three vectors $e_1,e_2,e_3$ of degree $+1$ and two vectors $h_1,h_2$ of degree $+2$.  Define 
the differential $d$  by setting $de_1=de_2=0$, $de_3=h_1$, and the nontrivial brackets $[e_i,e_j]$ as
\[ [e_1,e_1]=h_1,\quad [e_2,e_2]=h_2.\]

The DG-Lie algebra $M$ is formal. In fact, the graded subspace $I=\Span(e_3,h_1)$ is an acyclic differential ideal, hence the projection $M\to M/I$ is a quasi-isomorphism and 
the differential of $M/I$ is trivial.

Consider now the DG-Lie subalgebra $L=\Span(e_1,e_2+e_3,h_1,h_2)$. The inclusion map 
$L\to M$ does not satisfy the condition of Theorem~\ref{thm.main}. In fact, every graded algebraic complement must be of the form $N=\Span(\alpha e_1+\beta e_2+\gamma e_3)$ with $\beta\not=\gamma$;  the condition 
$[e_1,N]\subset N$ implies $\alpha=0$ and the condition $[e_2+e_3,N]\subset N$ implies 
$\beta=0$, a contradiction since the condition $d(N)\subset N$ implies $\gamma=0$.

We have $H^*(M)=\Span([e_1],[e_2],[h_2])$, $H^*(L)=\Span([e_1],[h_2])\subset H^*(M)$. Therefore $H^*(M)=H^*(L)\oplus\Span([e_2])$ and $[H^*(L),[e_2]]=0$ and the inclusion map 
$H^*(L)\to H^*(M)$ is split injective as a morphism of $H^*(L)$-modules.

In order to prove that $L$ is not formal, since its bracket is trivial in cohomology, it is sufficient to prove that its Maurer--Cartan functor is obstructed. Consider the element 
\[ u=te_1-\frac{t^2}{2}(e_2+e_3)\in L^1\otimes \K[t]/(t^4).\] 
Then $u\in MC_L(\K[t]/(t^4))$ since 
\[ du+\frac{1}{2}[u,u]=-\frac{t^2}{2}h_1+\frac{t^2}{2}h_1+\frac{t^4}{8}h_2=0.\]
Any lifting of $u$ to $L^1\otimes \K[t]/(t^5)$ is of type 
\[ v=te_1-\frac{t^2}{2}(e_2+e_3)+t^4(\alpha e_1+\beta(e_2+e_3)\]
and the coefficient of $h_2$ in 
$dv+\frac{1}{2}[v,v]$ is equal to $t^4/4$ in the ring $\K[t]/(t^5)$.

Notice that, if $0\not=\lambda\in \K$, then the DG-Lie algebra 
\[ L_{\lambda}=\Span(e_1+\lambda e_2,e_2+e_3,h_1,h_2)=\Span(e_1+\lambda e_2,e_3-e_1/\lambda,h_1,h_2)\]
is formal since $\Span(e_3-e_1/\lambda,h_1)$ is an acyclic differential graded  Lie ideal. In particular, without further assumptions, 
formality is not stable under specialisation.

\medskip\noindent
\textbf{Acknowledgements.} We thank Coline Emprin for pointing out an error in the proof of Lemma~5.3 of \cite{Ma15}.
M.M. is partially supported by the PRIN 20228JRCYB ``Moduli spaces and special varieties''  of  ``Piano Nazionale di Ripresa e Resilienza, Next Generation EU''.

\end{document}